\documentclass[reqno,11pt]{amsart}
\usepackage{hyperref}
\usepackage[latin1]{inputenc}
\usepackage{amsmath,amsthm}
\usepackage{amssymb}
\usepackage{mathabx}
\newtheorem{Theorem}{Theorem } [section]
\newtheorem{lemma}[Theorem]{Lemma}

\newtheorem{remark}{Remark}

\numberwithin{equation}{section}

\newcommand{\n}{|\hspace{-.5mm}\|}

\font\ff=cmsy10
\def\tiF{\text{\ff F\kern 0pt}{\;}^{ -1}}
\def\tF{\text{\ff F\kern 0pt}}

\begin{document}
\title[]{A note on the Ostrovsky equation in weighted Sobolev spaces}
\author{Eddye Bustamante, Jos\'e Jim\'enez Urrea and Jorge Mej\'{\i}a}
\thanks{Supported by Universidad Na\-cio\-nal de Co\-lom\-bia-Me\-de\-ll\'in and Colciencias, Colombia, project 111865842951.} 

\address{Eddye Bustamante M., Jos\'e Jim\'enez Urrea, Jorge Mej\'{\i}a L. \newline
Departamento de Matem\'aticas\\Universidad Nacional de Colombia\newline
A. A. 3840 Medell\'{\i}n, Colombia}
\email{eabusta0@unal.edu.co, jmjimene@unal.edu.co, jemejia@unal.edu.co}

\begin{abstract}
In this work we consider the initial value problem (IVP) associated to the Ostrovsky equations $$\left. \begin{array}{rl} u_t+\partial_x^3 u\pm \partial_x^{-1}u +u \partial_x u &\hspace{-2mm}=0,\qquad\qquad x\in\mathbb R,\; t\in\mathbb R,\\ u(x,0)&\hspace{-2mm}=u_0(x). \end{array} \right\}$$ We study the well-posedness of the IVP in the weighted Sobolev spaces $$Z_{s,\frac{s}2}:=\{u\in H^s(\mathbb R):D_x^{-s} u\in L^2(\mathbb R)\}\cap L^2(|x|^s dx ),$$ with $\frac34<s\leq 1$.
\end{abstract}

\maketitle

\section{Introduction}

In this article we consider the initial value problem (IVP) associated to the Ostrovsky equations,
\begin{align}
\left. \begin{array}{rl}
u_t+\partial_x^3 u\pm  \partial_x^{-1}u +u \partial_x u &\hspace{-2mm}=0,\qquad\qquad x\in\mathbb R,\; t\in\mathbb R,\\
u(x,0)&\hspace{-2mm}=u_0(x).
\end{array} \right\}\label{O}
\end{align}
The operator $\partial_x^{-1}$ in the equations  denotes a certain antiderivative with respect to the spatial variable $x$ defined through the Fourier transform by $(\partial_x^{-1}f)^{\widehat{}}:=\frac{\widehat{f}(\xi)}{i\xi}$. 

These equations were deducted in \cite{O1978} as a model for weakly nonlinear long waves, in a rotating  frame of reference, to describe the propagation of surface waves in the ocean. The sign of the third term of the equation is related to the type of dispersion. 

Linares and Milan\'es \cite{LM2006} proved that the IVP \eqref{O} for both equations is locally well-posed (LWP) for initial data $u_0$ in Sobolev spaces $H^s(\mathbb R)$, with $s>\frac34$, and such that $\partial_x^{-1}u_0\in L^2(\mathbb R)$. This result was obtained by the use of certain regularizing effects of the linear part of the equation. In \cite{IM2006}  and \cite{IM2007} Isaza and Mej\'\i a used Bourgain spaces and the technique of elementary calculus inequalities, introduced by Kenig, Ponce, and Vega in \cite{KPV1996}, to prove local well-posedness in Sobolev spaces $H^s(\mathbb R)$, with $s>-\frac34$, for both equations. Furthermore Isaza and Mej\'\i a, in \cite{IM2009}, established that the IVP \eqref{O}, for both equations, is not quantitatively well-posed in $H^s(\mathbb R)$ with $s<-\frac34$. Recently, Li et al. in \cite{L2015} proved that the IVP \eqref{O} with the minus sign is LWP in $H^{-3/4}(\mathbb R)$.

In \cite{K1983}, Kato studied the IVP for the generalized KdV equation in several spaces, besides the classical Sobolev spaces. Among them, Kato considered weighted Sobolev spaces.

In this work we will be concerned with the well-posedness of the IVP \eqref{O} in weighted Sobolev spaces. IThis type of spaces arises in a natural manner when we are interested in determining if the Schwartz space is preserved by the flow of the evolution equations in \eqref{O}. These spaces also appear in the study of the persistence in time of the regularity of the Fourier transform of the solutions of the IVP \eqref{O}.

Some relevant nonlinear evolution equations as the KdV equation, the non-linear Schrödinger equation, the Benjamin-Ono equation and the Zakharov-Kuznetsov equation have also been studied in the context of weighted Sobolev spaces (see \cite{FLP2012}, \cite{BJM2015}, \cite{FLP2013}, \cite{I1986}, \cite{I2003}, \cite{J2013}, \cite{N2012}, \cite{NP2009},  \cite{NP2011} and \cite{BJM2016} and references therein).

We will study real valued solutions of the IVP \eqref{O} in the weighted Sobolev spaces
$$Z_{s,r}:=\{u\in H^s(\mathbb R):D_x^{-s} u\in L^2(\mathbb R)\} \cap L^2(|x|^{2r} dx ),$$
with $s,r\in\mathbb R$ (here $D_x^{-s} u$ is defined through the Fourier transform by $(D_x^{-s} u)^{\widehat{}}:=|\xi|^{-s}\widehat{u}$). 

The Ostrovsky equations are perturbations of the KdV equation with the nonlocal term $\pm  \partial_x^{-1}u$. It is interesting to know how this term behaves when we are working in the context of weighted Sobolev spaces.

The relation between the indices $s$ and $r$ for the solutions of the IVP \eqref{O} can be found, following the considerations, contained in the work of Kato (for more details see \cite{BJM2016}): it turns out that the natural weighted Sobolev space to study the IVP \eqref{O} is $Z_{s,s/2}$.

Our aim in this article is to prove that the IVP \eqref{O} is locally well posed (LWP) in $Z_{s,s/2}$ for $\frac34<s\leq 1$. Our method of proof is based on the contraction mapping principle and has two ingredients. First of all, we use the result of local well posedness, obtained by Linares and Milan\'es in $X_s:=\{f\in H^s(\mathbb R):\partial_x^{-1}f\in L^2(\mathbb R)\}$, with $s>\frac34$. The statement of this result is as follows.

\textbf{Theorem A}. \textit{Let $u_0\in X_s$, $s>\frac34$. Then there exist $T=T(\|u_0\|_{H^s})>0$ and a unique solution $u$ of the IVP \eqref{O} such that
\begin{align}
u\in C([0,T];X_s)\,,\label{con1}\\
\|\partial_xu\|_{ L_T^4L_x^{\infty}}<\infty\,,\label{con2}\\
\|D_x^{s}\partial_x u\|_{L_x^{\infty}L_T^2}<\infty\,,\label{con3}\quad\text{and}\\
\|u\|_{L_x^2L_T^{\infty}}<\infty\,.\label{H1}
\end{align}
Moreover, for any $T'\in(0,T)$ there exists a neighborhood $V$ of $u_0$ in $X_s$ such that the map  datum-solution $\tilde{u}_0\mapsto\tilde{u}$ is Lipschitz from $H^s$ into the class defined by \eqref{con1}-\eqref{H1} with $T'$ instead of $T$.}

On the other hand, we need a tool to treat fractional powers of $|x|$. A key idea in this direction is to use a characterization of the generalized Sobolev space
\begin{align}
L^p_b(\mathbb R^n):=(1-\Delta)^{-b/2}L^p(\mathbb R^n),\label{intro2}
\end{align}
due to Stein (see \cite{S1961} and \cite{S1970}) (when $p=2$, $L^2_b(\mathbb R^n)=H^b(\mathbb R^n)$). This characterization is as follows.\\
\textbf{Theorem B}. \textit{Let $b\in(0,1)$ and $2n/(n+2b)\leq p<\infty$. Then $f\in L^p_b(\mathbb R^n)$ if and only if}
\begin{enumerate}
\item[(a)] $f\in L^p(\mathbb R^n)$, \textit{and}
\item[(b)] $\mathcal D^b f(x):=\left( \displaystyle\int_{\mathbb R^n} \dfrac{|f(x)-f(y)|^2}{|x-y|^{n+2b}} dy \right)^{1/2}\in L^p(\mathbb R^n)$,
\end{enumerate}
\textit{with}
\begin{align}
\| f\|_{L^p_b}:=\|(1-\Delta)^{b/2}f \|_{L^p}\simeq \| f\|_{L^p}+\| D^b f\|_{L^p}\simeq \| f\|_{L^p}+\| \mathcal D^b f\|_{L^p}, \label{intro3}
\end{align}
\textit{where $D^b f$ is the homogeneous fractional derivative of order $b$ of $f$, defined through the Fourier transform by}
\begin{align}
(D^b f)^\wedge(\xi) =|\xi|^b \hat f(\xi), \label{intro4}
\end{align}
\textit{($\xi\in\mathbb R^n$ is the dual Fourier variable of $x\in\mathbb R^n$).}\\
From now on we will refer to $\mathcal D^b f$ as the Stein derivative of $f$ of order $b$.\\
As a consequence of Theorem B, Nahas and Ponce proved (see Proposition 1 in \cite{NP2009}) that  for measurable functions $f,g:\mathbb R^n\to \mathbb C$,
\begin{align}
\mathcal D^b(fg) (x)\leq & \| f\|_{L^\infty} (\mathcal D^b g)(x)+|g(x)| \mathcal D^b f(x),\; a.e. \,x\in\mathbb R^n, \, \text{and} \label{intro5}\\
\| \mathcal D^b(fg)\|_{L^2} \leq &\| f\mathcal D^b g \|_{L^2}+ \|g\mathcal D^b f \|_{L^2}.\label{intro6}
\end{align} 
It is unknown whether or not \eqref{intro6} still holds with $D^b$ instead of $\mathcal D^b$.\\
Following a similar procedure to that done by Nahas and Ponce in \cite{NP2009}, in order to obtain a pointwise estimate for $\mathcal D^b(e^{it|x|^2})(x)$ (see Proposition 2 in \cite{NP2009}), we get to bound appropriately $\mathcal D^b(e^{itx^3})(x)$ and $\mathcal D^b(e^{\pm i\frac{t}{x}})(x)$ for $b\in (0,1/2]$ (see Lemmas \ref{lemastein1}  and \ref{lemastein2} in section 2).\\
Using \eqref{intro3} (for $p=2$), \eqref{intro5}, \eqref{intro6} and Lemmas \ref{lemastein1} and  \ref{lemastein2} we deduce an estimate for the weighted $L^2$-norm of the group associated to the linear part of the Ostrovsky equations, $\| |x|^b U_{\pm}(t)f\|_{L^2}$, in terms of $t$, $\||x|^b f\|_{L^2}$, $\| f\|_{H^{2b}}$ and $\|D^{-2b}f\|_{L^2}$ (see Lemma \ref{c2.7} in section 2).\\
This estimate is similar to that, obtained by Fonseca, Linares and Ponce in \cite{FLP2014} (see formulas 1.8 and 1.9 in Theorem 1) for the KdV equation.\\
The linear estimates for the group of the linear part of the Ostrovsky equations, obtained by Linares and Milan\'es in \cite{LM2006}, together with the estimate for the weighted $L^2$-norm of the group, allow us to obtain, by the contraction mapping principle, that the IVP \eqref{O} is LWP in a certain subspace of $Z_{s,s/2}$.

Now we formulate in a precise manner the main result of this article.

\begin{Theorem} \label{maint} Let $3/4<s\leq 1$ and $u_0\in Z_{s,s/2}$ a real valued function. Then there exist $T>0$ and a unique $u$, in a certain subspace $Y_T$ of $C([0,T];Z_{s,s/2})$, solution of the IVP \eqref{O}.
(The definition of the subspace $Y_T$ will be clear in the proof of the theorem).

Moreover, for any $T'\in(0,T)$ there exists a neighborhood $V$ of $u_0$ in $Z_{s,s/2}$ such that the data-solution map $\tilde u_0 \mapsto \tilde u$ from $V$ into $Y_{T'}$ is Lipschitz.

\end{Theorem}

\begin{remark} When $s>1$ we do not know an interpolation inequality, similar to  that in \cite{NP2009}, but including negative exponents of $(1-\Delta )$. By this reason  we can  not apply the method used in \cite{BJM2016} in order to obtain that the IVP \eqref{O} is LWP in $Z_{s,s/2}$ when $s>1$.\end{remark}

\begin{remark} Local well-posedness results of the IVP \eqref{O} in $H^s(\mathbb R)$ for $s\leq \frac 34$ use the context of Bourgain spaces. It is not clear for us how to handle our weights in those spaces. 
\end{remark}

This article is organized as follows: in section 2 we recall the Leibniz rule for fractional derivatives, deduced by Kenig, Ponce and Vega in \cite{KPV1993}  (subsection 2.1), and we find (subsection 2.2) appropriate estimates for the Stein derivatives of order $b$ in $\mathbb R$ of the functions $e^{itx^3}$ and $e^{\pm i\frac{t}{x}}$ (Lemmas \ref{lemastein1} and \ref{lemastein2}), which have an important consequence (Lemma \ref{c2.7}) that affirms that the weighted Sobolev space $Z_{s,s/2}$ remains invariant by the group. In section 3, we use the results, obtained in section 2, in order to prove Theorem \ref{maint}.

Throughout the paper the letter $C$ will denote diverse constants, which may change from line to line, and whose dependence on certain parameters is clearly established in all cases.

Finally, let us explain the notation for mixed space-time norms. For $f:\mathbb R\times [0,T]\to\mathbb R$ (or $\mathbb C$) we have
$$\|f \|_{L^p_x L^q_{T}}:=\left( \int_{\mathbb R} \left(   \int_0^T |f(x,t)|^q dt  \right)^{p/q} dx \right )^{1/p}.$$
When $p=\infty$ or $q=\infty$ we must do the obvious changes with \textit{essup}. Besides, when in the space-time norm appears $t$ instead of $T$, the time interval is $[0,+\infty)$.

\section{Preliminary Results}

\subsection{Leibniz rule}

In this subsection we recall the Leibniz rule for fractional derivatives, obtained in \cite{KPV1993}.

\begin{lemma}\label{l2.4} (Leibniz rule). Let us consider $0<\alpha<1$ and $1<p<\infty$. Thus
\begin{align*}
\| D^\alpha(fg)-f D^\alpha g-gD^\alpha f \|_{L^p(\mathbb R)}\leq C \|g \|_{L^\infty(\mathbb R)} \| D^\alpha f\|_{L^p(\mathbb R)}.
\end{align*}
\end{lemma}

\subsection{Stein derivative}

In this subsection, we obtain in Lemmas \ref{lemastein1} and \ref{lemastein2}  appropriate bounds for $\mathcal D^b(e^{itx^3})(x)$ and $\mathcal D^b(e^{\pm i\frac{t}{x}})(x)$, respectively. Then, using properties \eqref{intro5} and \eqref{intro6} of the Stein derivative and these Lemmas , we succeed, in Lemma \ref{c2.7}, to bound in  an adequate manner the weighted $L^2$-norm $\||x|^b U_{\pm}(t) f\|_{L^2_x}$, for the unitary groups $\{U_{\pm}(t)\}_{t\in\mathbb R}$ associated to the linear part of the Ostrovsky equations, i.e. $[U_{\pm}(t)f](x)$ is the solution of the PVI
\begin{align}
\left. \begin{array}{rl}
u_t+\partial_x^3 u\pm  \partial_x^{-1}u  &\hspace{-2mm}=0,\\
u(x,0)&\hspace{-2mm}=f(x),
\end{array} \right\}\label{OL}
\end{align}
and is given by
\begin{align}
[U_{\pm}(t)f](x)=\dfrac 1{\sqrt{2\pi}}\int_{\mathbb R} e^{i[t(\xi^3\pm\frac1{\xi})+x\xi]} \widehat f(\xi) d\xi .\label{group}
\end{align}
In the proofs of Lemmas \ref{lemastein1} and \ref{lemastein2} we will use repeatedly the following inequalities:
$$|e^{i\theta}-1|\leq 2\quad\text{and}\quad |e^{i\theta}-1|\leq |\theta|,\quad \text{for $\theta\in\mathbb R$}.$$
\begin{lemma}\label{lemastein1} Let $b\in (0,1)$. There exists a constant $C_b>0$ such that for any $t>0$ and $x\in\mathbb R$,
\begin{align*}
\mathcal D^b(e^{itx^3})(x)\leq C_b \left( t^{b/3} +t^{\frac13+\frac{2b}9} +(t^{\frac13+\frac{2b}3} +t^{\frac{2b}3}) |x|^{2b}\right).
\end{align*}
\end{lemma}
\begin{proof}
After the change of variables $w:=t^{1/3}(x-y)$ we have that
\begin{align}
\notag \mathcal D^b(e^{itx^3})(x)&=\left(\int_{\mathbb R} \dfrac{|e^{itx^3}-e^{ity^3}|^2}{|x-y|^{1+2b}}dy \right)^{1/2}\\
&=t^{b/3}\left(\int_{\mathbb R} \dfrac{|e^{i(-3x^2 t^{2/3}w_1+3xt^{1/3}w^2-w^3)}-1|^2}{|w|^{1+2b}}dw \right)^{1/2}\equiv t^{b/3} I. \label{st0}
\end{align}
Let us observe that
$$|i(-3x^2t^{2/3}w+3xt^{1/3}w^2-w^3)|\leq |w|(3x^2 t^{2/3}+3|x|t^{1/3}|w|+w^2).$$
In consequence, for $w$ such that $3x^2 t^{2/3}>3|x|t^{1/3}|w|$, i.e. for $w$ such that $|x|t^{1/3}>|w|$, it follows that
\begin{align}
\notag |-3x^2 t^{2/3}w+3x t^{1/3}w^2-w^3|&\leq |w| (6x^2t^{2/3}+|w|^2)\\
&\leq |w|(6x^2 t^{2/3}+x^2 t^{2/3})\leq 7x^2 t^{2/3}|w|.\label{st1}
\end{align}
In order to estimate $I$ we split the $\mathbb R$ line in three sets $E_i$, $i=1,2,3$.

First, we define
$$E_2:=\{w:|w|<t^{1/3}|x|,\; |w|<(t^{1/3}x^2)^{-1} \},$$
and we estimate 
$$\left(\int_{E_2} \dfrac{|e^{i(-3x^2 t^{2/3}w+3t^{1/3}w^2-w^3)}-1|^2}{|w|^{1+2b}}dw \right)^{1/2}.$$
Two cases will be consider to estimate this integral.

\textit{Case 2.1.} $t^{1/3}|x|\leq t^{-1/3} x^{-2}$.

In this case, taking into account \eqref{st1}, we have
\begin{align}
\notag \left(\int_{E_2} \dfrac{|e^{i(-3x^2 t^{2/3}w+3xt^{1/3}w^2-w^3)}-1|^2}{|w|^{1+2b}}dw \right)^{1/2}&\leq C x^2 t^{2/3} \left(  \int_0^{t^{1/3}|x|} \dfrac{w^2}{w^{1+2b}}dw \right)^{1/2}\\
\notag &=C x^2 t^{2/3} \left( \dfrac{   (t^{1/3} |x|)^{2-2b}}   {2-2b} \right)^{1/2}\\
 &= C_b |x|^{3-b} t^{1-b/3}\leq C_b t^{1/3-b/9},\label{st1b}
\end{align}
where in the last inequality the condition $|x|^3< t^{-2/3}$ was used.

\textit{Case 2.2.} $t^{1/3}|x|> t^{-1/3} x^{-2}$.

A simple calculation shows that
\begin{align}
\notag \left(\int_{E_2} \dfrac{|e^{i(-3x^2 t^{2/3}w+3xt^{1/3}w^2-w^3)}-1|^2}{|w|^{1+2b}}dw \right)^{1/2}&\leq C x^2 t^{2/3} \left(\int_0^{t^{-1/3}x^{-2}} w^{1-2b} dw \right)^{1/2}\\
&\leq C_b t^{1/3+b/3} |x|^{2b}. \label{st1c}
\end{align}
From \eqref{st1b} and \eqref{st1c} we have that
\begin{align}
\left(\int_{E_2} \dfrac{|e^{i(-3x^2 t^{2/3}w+3xt^{1/3}w^2-w^3)}-1|^2}{|w|^{1+2b}}dw \right)^{1/2}\leq C (t^{1/3-b/9}+t^{1/3+b/3}|x|^{2b}).\label{st2}
\end{align}
For the set
$$E_1:=\{ w: |w|> (t^{1/3} x^2)^{-1}\},$$
one has
\begin{align}
\notag&\left(\int_{E_1} \dfrac{|e^{i(-3x^2 t^{2/3}w+3xt^{1/3}w^2-w^3)}-1|^2}{|w|^{1+2b}}dw \right)^{1/2}\leq 2\left(\int_{E_1} \dfrac 1{|w|^{1+2b}} dw \right)^{1/2}\\
&\leq C \left( \int_{(t^{1/3}x^2)^{-1}}^\infty \dfrac 1{w^{1+2b}} dw \right)^{1/2}\leq C_b (t^{-1/3} x^{-2})^{-b}=C_b t^{b/3}|x|^{2b}.\label{st3}
\end{align}
From \eqref{st2} and \eqref{st3}, if $\min\{t^{1/3} |x|,(t^{1/3}x^2)^{-1}\}=(t^{1/3}x^2)^{-1}$, we obtain that
\begin{align}
\left(\int_{\mathbb R} \dfrac{|e^{i(-3x^2 t^{2/3}w+3xt^{1/3}w^2-w^3)}-1|^2}{|w|^{1+2b}}dw \right)^{1/2}\leq C_b [t^{1/3-b/9}+(t^{1/3+b/3}+t^{b/3})|x|^{2b}].\label{st4}
\end{align}
Now we consider the case $\min\{t^{1/3} |x|,(t^{1/3}x^2)^{-1}\}=t^{1/3}|x|$, i.e. $|x|^3 t^{2/3}<1$, and for that purpose we define
$$E_3:=\{ w: t^{1/3} |x|<|w|< (t^{1/3} x^2)^{-1}\}.$$
In order to estimate
$$\left(\int_{E_3} \dfrac{|e^{i(-3x^2 t^{2/3}w+3xt^{1/3}w^2-w^3)}-1|^2}{|w|^{1+2b}}dw \right)^{1/2},$$
we need to consider three cases.

\textit{Case 3.1.} $1<t^{1/3}|x|$.

For this case we note that
\begin{align}
\notag \left(\int_{E_3} \dfrac{|e^{i(-3x^2 t^{2/3}w+3xt^{1/3}w^2-w^3)}-1|^2}{|w|^{1+2b}}dw \right)^{1/2} \leq &C \left( \int_{E_3} \dfrac 1{|w|^{1+2b}} \right)^{1/2}\leq C\left( \int_{t^{1/3}|x|}^{(t^{1/3}x^2)^{-1}} \dfrac 1{w^{1+2b}} dw\right)^{1/2}\\
\notag &= C_b [(t^{1/3}|x|)^{-2b}-(t^{1/3}x^2)^{2b}]^{1/2}\leq C_b (t^{1/3}|x|)^{-b}\\
&\leq C_b.\label{st4a}
\end{align}
\textit{Case 3.2.} $t^{1/3}|x|<1<(t^{1/3}x^2)^{-1}$.

Let us observe that for $|w|<1$,
\begin{align}
|w(-3x^2t^{2/3}+3xt^{1/3}w-w^2)|\leq |w| (3+3|w|+w^2)\leq C|w|,\label{st8as}
\end{align}
and then
\begin{align}
\notag\left(\int_{E_3} \dfrac{|e^{i(-3x^2 t^{2/3}w+3xt^{1/3}w^2-w^3)}-1|^2}{|w|^{1+2b}}dw \right)^{1/2} &\leq C \left(  \int_{t^{1/3}|x|}^1 \dfrac{ w^2}{w^{1+2b}} dw +   \int_1^{(t^{1/3}x^2)^{-1}} \dfrac 1{w^{1+2b}} dw \right)^{1/2}\\
&= C_b [1-(t^{1/3}|x|)^{2-2b}+1-(t^{1/3}x^2)^{2b}]^{1/2}\leq C_b.\label{st8a}
\end{align}
\textit{Case 3.3.} $(t^{1/3}x^2)^{-1}<1$.

In this final case we obtain, using \eqref{st8as},
\begin{align}
\notag \left(\int_{E_3} \dfrac{|e^{i(-3x^2 t^{2/3}w+3xt^{1/3}w^2-w^3)}-1|^2}{|w|^{1+2b}}dw \right)^{1/2} &\leq  C \left(  \int_{t^{1/3}|x|}^{(t^{1/3}x^2)^{-1}} \dfrac{w^2}{w^{1+2b}}dw \right)^{1/2}\leq C_b [(t^{1/3}x^2)^{-1}]^{1-b}\\
&\leq C_b.\label{st12a}
\end{align}
Consequently, from \eqref{st4a}, \eqref{st8a} and \eqref{st12a}, in any case, for $b\in(0,1)$,
\begin{align}
\left(\int_{E_3} \dfrac{|e^{i(-3x^2 t^{2/3}w+3xt^{1/3}w^2-w^3)}-1|^2}{|w|^{1+2b}}dw \right)^{1/2}\leq C_b. \label{st17}
\end{align}
Summarizing, estimates \eqref{st0}, \eqref{st4} and \eqref{st17} imply that, for $b\in(0,1)$,
\begin{align*}
\notag \mathcal D^b(e^{itx^3})(x) & \leq C_b t^{b/3}(1+t^{1/3-b/9}+(t^{1/3+b/3}+t^{b/3})|x|^{2b})\\
& \leq C_b [t^{b/3}+t^{1/3+2b/9}+(t^{1/3+2b/3}+t^{2b/3})|x|^{2b}]
\end{align*}
Lemma \ref{lemastein1} is proved.
\end{proof}
\begin{remark} Lemma \ref{lemastein1} can be used in order to obtain well-posedness results for the IVP associated to the KdV equation in weighted Sobolev spaces. 
\end{remark}
\begin{lemma}\label{lemastein2} Let $b\in (0,\frac12]$. There exists a constant $C_b>0$ such that for any $t>0$ and $x\ne 0$,
\begin{align*}
\mathcal D^b(e^{\pm i\frac{t}{x}})(x)\leq C_b\frac{t^b}{|x|^{2b}} .
\end{align*}
\end{lemma}
\begin{proof}
We only consider the case with the minus sign, being the other one similar. Without loss of generality we suppose $x>0$. We need to consider two cases.

\textit{Case 1.} $0<t/x\leq 6 \pi$.

After the change of variables $w:=x-y$ we obtain
$$\mathcal D^b(e^{- i\frac{t}{x}})(x)= \left(\int_{\mathbb R} \dfrac{|e^{i f(w)}-1|^2}{|w|^{1+2b}} dw \right)^{1/2},$$
where $f(w)=\dfrac tx \left(\dfrac w{w-x}\right)$. Let us define
$$E_1:=\{w:|w|\leq \frac x2 \}.$$
If $w\in E_1$, then $|w-x|\geq \frac x2$. Thus
$$|f(w)|=\dfrac tx \left(\dfrac {|w|}{|w-x|}\right)\leq \dfrac {2t}{x^2}|w|,$$
and
\begin{align}
\notag \left(\int_{E_1} \dfrac{|e^{i f(w)}-1|^2}{|w|^{1+2b}} dw \right)^{1/2}&\leq \left(\int_{E_1} \dfrac{|f(w)|^2}{|w|^{1+2b}} dw \right)^{1/2}\leq 2 \dfrac t{x^2} \left( \int_{-x/2}^{x/2} \dfrac{w^2}{|w|^{1+2b}} dw \right)^{1/2}\\
&\leq C_b \dfrac t{x^2} x^{1-b}=C_b\dfrac{t^b}{x^{2b}}\left(\dfrac{t^{1-b}}{x^{1-b}}\right)\leq C_b\dfrac{t^b}{x^{2b}}.\label{l2.3.1}
\end{align}
Let us observe that $f(-\frac x2)=\frac 13 \frac tx$. Hence, for $w\leq -\frac x2$, $\frac 13\frac tx\leq f(w)<\frac tx$.\\
We define now the set
$$E_2:=\{w:w\leq -\frac x2 \}.$$
Then
\begin{align}
\notag \left(\int_{E_2} \dfrac{|e^{i f(w)}-1|^2}{|w|^{1+2b}} dw \right)^{1/2}&\leq \left(\int_{E_2} \dfrac{|f(w)|^2}{|w|^{1+2b}} dw \right)^{1/2}\leq \dfrac t{x} \left( \int_{-\infty}^{-x/2} \dfrac{1}{|w|^{1+2b}} dw \right)^{1/2}\\
&\leq C_b \dfrac t{x} x^{-b}=C_b\dfrac{t^b}{x^{2b}}\left(\dfrac{t^{1-b}}{x^{1-b}}\right)\leq C_b\dfrac{t^b}{x^{2b}}.\label{l2.3.2}
\end{align}
Taking into account that $f(\frac 32x)=3\frac tx$, then, for $w\geq \frac 32x$, $\frac tx<f(w)\leq 3\frac tx$.\\
Let us define
$$E_3:=\{ w:w\geq \frac 32x\}.$$
Then, in a similar way as it was done in the estimation over the set $E_2$, we obtain
\begin{align}
\left(\int_{E_3} \dfrac{|e^{i f(w)}-1|^2}{|w|^{1+2b}} dw \right)^{1/2}\leq C_b \dfrac {t^b}{x^{2b}}.\label{l2.3.3}
\end{align}
Let us consider the sequence $\{ w_n\}$ such that $f(w_n)=3\frac tx+2(n-1)\pi$. This is a decreasing sequence such that $w_1=\frac 32 x$,
$$w_n=\dfrac{3t+2(n-1)\pi x}{2\dfrac tx+2(n-1)\pi},$$
and $\underset{n\to\infty}{\lim} w_n=x$. For $n\geq 1$ and $w\in (w_{n+1},w_n)$ we have that $f(w_n)<f(w)<f(w_{n+1})$, i.e.,
$$3\dfrac tx+2(n-1)\pi<f(w)<3\dfrac tx+2n\pi.$$
From these inequalities, it is easy to see that
$$f(w)-2(n-1)\pi<\dfrac{3\dfrac tx+2\pi}{3\dfrac tx+2(n-1)\pi}f(w).$$
Let us define
$$E_4:=\{w:x<w<w_2 \}.$$
Taking into account that
$$\int_{w_{n+1}}^{w_n} \dfrac 1{(w-x)^2}dw=\frac{2\pi}t,$$
it follows that
\begin{align}
\notag& \left(\int_{E_4} \dfrac{|e^{i f(w)}-1|^2}{|w|^{1+2b}} dw \right)^{1/2}\\
\notag&=\left(\sum_{n=2}^{\infty}\int_{w_{n+1}}^{w_n} \dfrac{|e^{i(f(w)-2(n-1)\pi)}-1|^2}{|w|^{1+2b}}dw\right)^{1/2}\leq \left(\sum_{n=2}^{\infty}\int_{w_{n+1}}^{w_n} \dfrac{|(f(w)-2(n-1)\pi)|^2}{|w|^{1+2b}}dw\right)^{1/2}\\
\notag&\leq \left(\sum_{n=2}^\infty \dfrac{(3\dfrac tx+2\pi)^2}{(3\dfrac tx+2(n-1)\pi)^2} \int_{w_{n+1}}^{w_n}\dfrac{f(w)^2}{|w|^{1+2b}}dw\right)^{1/2}\leq C \left( \sum_{n=2}^\infty\dfrac 1{(n-1)^2}\dfrac {t^2}{x^2}\int_{w_{n+1}}^{w_n}\dfrac {w^2}{w^{1+2b}(w-x)^2}dw\right)^{1/2}\\
\notag&\leq C \dfrac tx  \left( \sum_{n=2}^\infty\dfrac 1{(n-1)^2}\int_{w_{n+1}}^{w_n}\dfrac {w^{1-2b}}{(w-x)^2}dw\right)^{1/2}\leq C\frac tx \left(\sum_{n=2}^\infty \dfrac 1{(n-1)^2} w_n^{1-2b}\int_{w_{n+1}}^{w_n} \dfrac{dw}{(w-x)^2}\right)^{1/2}\\
&\leq C\dfrac tx x^{1/2-b} \left( \sum_{n=2}^\infty \dfrac 1{(n-1)^2} \dfrac{2\pi}t\right)^{1/2}=C\dfrac tx x^{1/2-b}\dfrac 1{t^{1/2}}=C\dfrac{t^{1/2}}{x^{1/2+b}}=C\dfrac{t^b}{x^{2b}}\dfrac{t^{1/2-b}}{x^{1/2-b}}\leq C\dfrac{t^b}{x^{2b}}.\label{l2.3.4}
\end{align}
Let us define
$$E_5:=\{w:w_2\leq w<w_1=\frac 32x \}.$$
Since for $w\in(w_2,w_1)$, $w\sim x$, we have
\begin{align}
\notag\left(\int_{E_5} \dfrac{|e^{i f(w)}-1|^2}{|w|^{1+2b}} dw \right)^{1/2}&\leq\left(\int_{w_2}^{w_1} \dfrac{f(w)^2}{w^{1+2b}} dw\right)^{1/2}\leq \dfrac tx \left(\int_{w_2}^{w_1} \dfrac{w^2}{w^{1+2b}(w-x)^2} \right)^{1/2}\\
\notag&\leq \dfrac tx \left(\int_{w_2}^{w_1} \dfrac{w^{1-2b}}{(w-x)^2} \right)^{1/2}\leq C\dfrac tx x^{1/2-b} \left(\int_{w_2}^{w_1} \dfrac 1{(w-x)^2} dw \right)^{1/2}\\
&=C\dfrac tx x^{1/2-b}\left(\dfrac {2\pi} t \right)^{1/2}=C\dfrac{t^b}{x^{2b}}\left(\dfrac{t^{1/2-b}}{x^{1/2-b}}\right)\leq C\dfrac{t^b}{x^{2b}}.\label{l2.3.5}
\end{align}
Finally, we define
$$E_6:=\{w:x/2< w<x \}.$$
Then, in a similar way as the estimation over the sets $E_4$ and $E_5$, we obtain
\begin{align}
\left(\int_{E_6} \dfrac{|e^{i f(w)}-1|^2}{|w|^{1+2b}} dw \right)^{1/2}\leq C\dfrac{t^b}{x^{2b}}.\label{l2.3.6}
\end{align}
Hence, from \eqref{l2.3.1} to \eqref{l2.3.6}, it follows that
$$\mathcal D^b(e^{-it/x})(x)=\left(\int_{\mathbb R} \dfrac{|e^{i f(w)-1}|^2}{|w|^{1+2b}}dw \right)^{1/2}\leq C_b\dfrac{t^b}{x^{2b}}.$$
\textit{Case 2.} $t/x> 6 \pi$.\\
Taking into account the change of variables $w'=1/w$ we obtain
$$\mathcal D^b(e^{-it/x})(x)=\left(\int_{\mathbb R}\dfrac{|e^{if(w)}-1|^2}{|w|^{1+2b}}dw \right)^{1/2}=\left(\int_{\mathbb R}\dfrac{|e^{ig(w)}-1|^2}{|w|^{1-2b}} dw\right)^{1/2},$$
where $g(w):=\frac tx\frac 1{1-xw}$ is an increasing function in $(-\infty,1/x)$ and in $(1/x,\infty)$. Besides
$$\lim_{w\to\pm\infty} g(w)=0,\quad \lim_{w\to \frac1x^-} g(w)=\infty\quad \text{and}\quad \lim_{w\to\frac 1x^+}g(w)=-\infty.$$
Let us define $\tilde w_1$ and $\tilde w_2$ by $g(\tilde w_1):=2\pi$ and $g(\tilde w_2):=-2\pi$. Then
$$\tilde w_1=\dfrac 1x-\dfrac t{2\pi x^2}\quad\text{and}\quad \tilde w_2=\dfrac 1x+\dfrac t{2\pi x^2}.$$
We split $\mathbb R-\{1/x \}$ in several sets $F_i$ as follows:
$$F_1:=(-\infty, \tilde w_1],\quad F_2:=(\tilde w_1,1/x),\quad F_3:=(1/x,\tilde w_2],\quad F_4:=(\tilde w_2,\infty).$$
Then
\begin{align*}
\left(\int_{F_1} \dfrac{|e^{i g(w)}-1|^2}{|w|^{1-2b}} dw \right)^{1/2}&\leq \left( \int_{F_1} \dfrac{|g(w)|^2}{|w|^{1-2b}}dw\right)^{1/2}=\left(\int_{-\infty}^{1/x-t/(2\pi x^2)} \dfrac{t^2/x^2}{|w|^{1-2b}(1-xw)^2} dw\right)^{1/2}\\
&\leq \dfrac tx \left( \int_{t/(2\pi x^2)-1/x}^\infty \dfrac 1{w^{1-2b}x^2 w^2}dw\right)^{1/2}\leq C_b \dfrac t{x^2} \left(\dfrac t{2\pi x^2}-\dfrac 1x \right)^{-1+b}\\
&=C_b \dfrac t{x^2} x^{1-b}\left( \dfrac tx-2\pi\right)^{-1+b}.
\end{align*}
Since $t/x>6\pi$, we have that $t/x-2\pi\sim t/x$, and in consequence
\begin{align}
\left(\int_{F_1} \dfrac{|e^{i g(w)}-1|^2}{|w|^{1-2b}} dw \right)^{1/2}&\leq C_b\dfrac t{x^2}x^{1-b}\dfrac{t^{-1+b}}{x^{-1+b}}=C_b\dfrac{t^b}{x^{2b}}\label{l2.3.7}
\end{align}
On the other hand, taking into account that
$$\dfrac t{2\pi x^2}-\dfrac 1x>\dfrac 2x,$$
we have
\begin{align}
\notag \left(\int_{F_2} \dfrac{|e^{i g(w)}-1|^2}{|w|^{1-2b}} dw \right)^{1/2}&\leq C\left[\left(\int_{1/x-t/(2\pi x^2)}^0 \dfrac 1{|w|^{1-2b}} dw\right)^{1/2}+\left(\int_0^{1/x}\dfrac 1{w^{1-2b}} dw \right)^{1/2} \right]\\
&\leq C \left( \int_0^{t/(2\pi x^2)-1/x} \dfrac 1{w^{1-2b}} dw \right)^{1/2}\leq C_b\left(\dfrac t{2\pi x^2}-\dfrac 1x \right)^b\leq C_b \dfrac {t^b}{x^{2b}}.\label{l2.3.8}
\end{align}
Proceeding in a similar manner as it was done in the sets $F_2$ and $F_1$, respectively, it can be proved that
\begin{align}
\left(\int_{F_i} \dfrac{|e^{i g(w)}-1|^2}{|w|^{1-2b}} dw \right)^{1/2}\leq C_b \dfrac{t^b}{x^{2b}},\quad i=3,4.\label{l2.3.8}
\end{align}
From \eqref{l2.3.7} to \eqref{l2.3.8} we conclude that for $t>0$ and $x>0$,
$$\mathcal D^b(e^{-i\frac tx})(x)=\left(\int_{\mathbb R} \dfrac{|e^{ig(w)}-1|^2}{|w|^{1-2b}} dw\right)^{1/2}\leq C_b\dfrac{t^b}{x^{2b}}.$$
\end{proof}

\begin{lemma}\label{c2.7} Let $\{U_{\pm}(t)\}_{t\in\mathbb R}$  be the group defined by \eqref{group}. For $b\in (0,1/2]$, there exists $C_b>0$ such that for $t\geq 0$ and $f\in Z_{2b,b}$,
\begin{align}
\notag\| &|x|^b U_{\pm}(t) f\|_{L^2_{x}}\\
&\leq C_b [(1+t^{b/3}+t^{\frac13+\frac{2b}9}) \| f\|_{L^2_{x}}+(t^{\frac13+\frac{2b}3}+t^{2b/3}) \|D^{2b}f \|_{L^2_{x}}+ t^b\|D^{-2b}f\|_{L^2_x}+\| |x|^b f \|_{L^2_{x}}].\label{st21}
\end{align}
\end{lemma}
\begin{proof}
Taking into account the definition of $D^b$ (see \eqref{intro4}), Plancherel's theorem, the properties \eqref{intro3}, \eqref{intro6} and \eqref{intro5} of the Stein derivative  $\mathcal D^b$, and Lemmas \ref{lemastein1} and  \ref{lemastein2} and using the notation $\text{}^{\vee}$ for the inverse Fourier transform, we have:
\begin{align*}
\| (|x|&^b U_{\pm}(t)f \|_{L^2}\\
=& \| |x|^b (e^{it(\xi^3\pm\frac1{\xi})} \hat f) ^\vee\|_{L^2_{x}}=\| |-x|^b (e^{it(\xi^3\pm\frac1{\xi})} \hat f) ^\wedge (-x)\|_{L^2_{x}}\\
\leq& C\| [D^b(e^{it(\xi^3\pm\frac1{\xi})} \hat f)]^\vee (x) \|_{L^2_{x}}\leq C\left( \| e^{it(\xi^3\pm\frac1{\xi})} \hat f \|_{L^2_{\xi}}+\| \mathcal D^b (e^{it(\xi^3\pm\frac1{\xi})} \hat f) \|_{L^2_{\xi}} \right)\\
\leq& C\left( \|f \|_{L^2_{x}}+\|\hat f \mathcal D^b (e^{it(\xi^3\pm\frac1{\xi})}) \|_{L^2_{\xi}}+\|\mathcal D^b (\hat f) e^{it(\xi^3\pm\frac1{\xi})}\|_{L^2_{\xi}} \right)\\
\leq& C \left( \|f \|_{L^2_{x}}+\|\hat f ( \mathcal D^b (e^{it\xi^3})+\mathcal D^b (e^{\pm i\frac{t}{\xi}}) )\|_{L^2_{\xi}}+\|\mathcal D^b (\hat f) \|_{L^2_{\xi}} \right)\\
\leq & C \left( \|f \|_{L^2_{x}}+C_b \|\hat f (t^{b/3}+t^{\frac13+\frac{2b}9}+(t^{\frac13+\frac{2b}3}+t^{\frac{2b}3})|\xi|^{2b}+t^b|\xi|^{-2b})\|_{L^2_{\xi}}+C(\|\hat f\|_{L^2_{\xi}}+\|D^b\hat f\|_{L^2_{\xi}}) \right)\\
\leq& C\|f\|_{L^2_x}+C_b (t^{b/3}+t^{\frac13+\frac{2b}9}) \|f \|_{L^2_x}+C_b(t^{\frac13+\frac{2b}3}+t^{2b/3}) \| D^{2b} f\|_{L^2_x}+ C_bt^b\|D^{-2b}f\|_{L^2_x}+C\| |x|^b f \|_{L^2_x} .
\end{align*}
\end{proof}

\section{Proof of the main theorem}

\begin{proof}
We consider the equivalent integral formulation of the IVP \eqref{O}
\begin{align}
u(t)=U_{\pm}(t)u_0- \int_0^t U_{\pm}(t-t')(u\partial_x u)(t') dt'.\label{proint}
\end{align}
 Let us define the integral operator
\begin{align}
\Psi(v)(t):=U_{\pm}(t)u_0- \int_0^t U_{\pm}(t-t')(v\partial_x v)(t') dt'.\label{mt3.2}
\end{align}
Proceeding as in \cite{LM2006}, let us define, for $T>0$, the metric space
\begin{align}
X_T:=\{v\in C([0,T];X_s):  \n v \n <\infty \}, \label{mt3.4}
\end{align}
where
\begin{align}
\notag \n v \n :=& \| v\|_{L_T^\infty H^s(\mathbb R)}+ \|\partial_x^{-1}v\|_{L_T^{\infty}L_x^2}+
\|\partial_xv\|_{L^4_T L^\infty_{x}}+\|D^s\partial_xv\|_{L_x^\infty L_T^2}+\|v \|_{L^2_{x}L^\infty_{T}}+\|v \|_{L^\infty_T L^2(|x|^s dx)}\\
&\equiv \sum_{i=1}^{6} n_i(v).\label{mt3.5}
\end{align}
(When $s=1$ in \eqref{mt3.5} we change $D^s$  by $\partial_x$)

For $a>0$, let $X_T^a$ be the closed ball in $X_T$ defined by
\begin{align}
X_T^a:= \{ v\in X_T: \n v \n\leq a  \}. \label{mt3.6}
\end{align}

We will prove that there exist $T>0$ and $a>0$ such that the operator $\Psi:X_T^a\to X_T^a$ is a contraction.

From now on we will suppose as in \cite{LM2006} that we are working with the group $\{U_+(t)\}$, being the other case similar.

In \cite{LM2006} it was proven that
\begin{align}\sum_{i=1}^{5} n_i(\Psi(v))\leq C(1+T)^{\frac12}\|u_0\|_{X_s}+CT^{\frac12}((1+T)^{\frac12}(1+T^{\frac14}+T^{\frac12})+(1+T^{\frac14})(1+T^{\frac12}))\n v\n^2\,.\label{LM}
\end{align}
Let us estimate $n_6(\Psi(v))$. For $t\in[0,T]$ and $\frac34<s\leq 1$,
applying Lemma \ref{c2.7} in section 2.2 with $b:=\frac{s}2$ we have that
\begin{align}
\notag\|\Psi(v)(t) \|&_{L^2(|x|^s dx)}\\
\notag\leq& \| U_+(t)u_0\|_{L^2(|x|^s dx)}+C\int_0^t\| U_+(t-t')(vv_x)(t') \|_{L^2(|x|^s dx)}dt'\\
\notag\leq& C_s \left[(1+t^{s/6}+t^{\frac13+\frac{s}9})\| u_0\|_{L^2_{x}}+(t^{1/3+s/3}+t^{s/3}) \|D^s u_0 \|_{L^2_{x}}+t^{\frac{s}2}\|D^{-s}u_0\|_{L^2_x}+\||x|^{\frac{s}2}u_0\|_{L^2_x}\right]\\
\notag&+C\int_0^t  C_s (1+(t-t')^{s/6}+(t-t')^{(\frac13+\frac{s}9)})\| (v v_x)(t')\|_{L^2_x}dt'\\
\notag&+C\int_0^tC_s((t-t')^{1/3+s/3}+(t-t')^{s/3})  \| D^s(vv_x)(t')\|_{L^2_x}dt'\\
\notag&+C\int_0^tC_s[(t-t')^{\frac{s}2}\|D^{-s}(vv_x)(t')\|_{L^2_x}+\||x|^{s/2}(vv_x)(t') \|_{L^2_x} ] dt'\\
\notag\leq& C_s \left[(1+T^{\frac13+\frac{s}3})\left(\| u_0\|_{L^2_{x}}+ \|D^s u_0 \|_{L^2_{x}}+\|D^{-s}u_0\|_{L^2_x}\right)+\||x|^{\frac{s}2}u_0\|_{L^2_x}\right]\\
\notag&+C_s(1+T^{\frac13+\frac{s}3})\int_0^T [ \| (v v_x)(t')\|_{L^2_x}+\| D^s(vv_x)(t')\|_{L^2_x}+\|D^{-s}(vv_x)(t')\|_{L^2_x}]dt'\\
&+C\int_0^T\||x|^{s/2}(vv_x)(t') \|_{L^2_x}  dt'\,.\label{LM1}
\end{align}
Taking into account that  for $0<s\leq 1$, it follows that
\[\|D^{-s}f\|_{L^2_x}\leq\|\partial_x^{-1}f\|_{L^2_x}+\|f\|_{L^2_x}\,,\]
we can conclude from \eqref{LM1} that
\begin{align}
\notag\|\Psi(v)(t) \|&_{L^2(|x|^s dx)}\\
\notag\leq& C_s\left[ (1+T^{\frac13+\frac{s}3})\|u_0\|_{X_s}+\||x|^{\frac{s}2}u_0\|_{L^2_x}\right]\\
&+C_s(1+T^{\frac13+\frac{s}3})\int_0^T [ \| (v v_x)(t')\|_{X_s}dt'+C\int_0^T\||x|^{s/2}(vv_x)(t') \|_{L^2_x}  dt'.\label{LM2}
\end{align}
Since
\begin{align*}\int_0^T\||x|^{s/2}(vv_x)(t') \|_{L^2_x}  dt'&\leq T^{\frac12}\||x|^{s/2}(vv_x) \|_{L^2_TL^2_x}\leq T^{\frac12}\||x|^{\frac{s}2}v\|_{L^{\infty}_TL^2_x}\|v_x\|_{L^2_TL^{\infty}_x}\\
&\leq T^{\frac12}\||x|^{\frac{s}2}v\|_{L^{\infty}_TL^2_x}T^{\frac14}\|v_x\|_{L^4_TL^{\infty}_x}\,,
\end{align*}
from \eqref{LM2} it follows that
\begin{align}
\notag n_6(\Psi(v)) \leq &C_s\left[ (1+T^{\frac13+\frac{s}3})\|u_0\|_{X_s}+\||x|^{\frac{s}2}u_0\|_{L^2_x}\right]\\
\notag&+C_s(1+T^{\frac13+\frac{s}3})\int_0^T [ \| (v v_x)(t')\|_{X_s}dt'+C T^{\frac34}\||x|^{\frac{s}2}v\|_{L^{\infty}_TL^2_x}\|v_x\|_{L^4_TL^{\infty}_x}\\
\notag\leq&C_s\left[ (1+T^{\frac13+\frac{s}3})\|u_0\|_{X_s}+\||x|^{\frac{s}2}u_0\|_{L^2_x}\right]\\
&+C_s(1+T^{\frac13+\frac{s}3})\int_0^T [ \| (v v_x)(t')\|_{X_s}dt'+C T^{\frac34}\n v\n^2.\label{LM3}
\end{align}
Using Cauchy-Schwarz's inequality and Leibniz's rule (Lemma \ref{l2.4}) in \cite{LM2006} it was proved that
\[\int_0^T [ \| (v v_x)(t')\|_{X_s}dt'\leq C(T+T^{\frac34}+T^{\frac12})\n v\n^2\,.\]
Then from \eqref{LM3} we obtain
\begin{align}
\notag n_6(\Psi(v)) \leq &C_s\left[ (1+T^{\frac13+\frac{s}3})\|u_0\|_{X_s}+\||x|^{\frac{s}2}u_0\|_{L^2_x}\right]\\
&+C_s(1+T^{\frac13+\frac{s}3})(T+T^{\frac34}+T^{\frac12})\n v\n^2.\label{LM4}
\end{align}

From estimates \eqref{LM} and \eqref{LM4}, taking into account that  $s>\frac34$, we conclude that
\begin{align}
\n \Psi(v)\n \leq C_s\left[ (1+T^{\frac13+\frac{s}3})\|u_0\|_{X_s}+\||x|^{\frac{s}2}u_0\|_{L^2_x}\right]
+C_sT^{\frac12}(1+T^{\frac13+\frac{s}3})(1+T^{\frac14}+T^{\frac12}) \n v \n^2.\label{mt3.32}
\end{align}

If we choose
$$a:= 2C_s\left[ (1+T^{\frac13+\frac{s}3})\|u_0\|_{X_s}+\||x|^{\frac{s}2}u_0\|_{L^2_x}\right],$$
and $T>0$ such that
$$C_s T^{\frac12}(1+T^{\frac13+\frac{s}3})(1+T^{\frac14}+T^{\frac12})a<1/2,$$
it can be seen that $\Psi$ maps $X^a_T$ into itself. Moreover, for $T$ small enough, $\Psi:X_T^a\to X_T^a$ is a contraction. In consequence, there exists a unique $u\in X_T^a$ such that $\Psi(u)=u$. In other words, for $t\in [0,T]$,
$$u(t)=U_+(t)u_0- \int_0^t U_+(t-t')(u\partial_x u)(t') dt',$$
i.e., the IVP \eqref{O} has a unique solution in $X_T^a$.

Using standard arguments, it is possible to show that for any $T'\in(0,T)$ there exists a neighborhood $V$ of $v_0$ in $Z_{s,s/2}$ such that the map $\tilde u_0\to \tilde v$ from $V$ into the metric space $X_{T'}$ is Lipschitz. Then the assertion of Theorem 1.1 follows if we take
$$Y_T:=\{ u\in C([0,T];Z_{s,s/2}): \n v\n <\infty \}.$$
\end{proof}

\end{document}